\documentclass[11pt]{amsart}
\usepackage{amssymb}
\usepackage[normalem]{ulem}
\usepackage{hyperref}

\title[The \c{S}i\c{s} Kebap Theorem]{The \c{S}i\c{s} Kebap Theorem:\\ A Generic Identification Theorem
for Groups\\ of Finite Morley Rank}

\author[A. Berkman]{Ay\c{s}e Berkman}
\address{Mathematics Department, Mimar Sinan University, Silahsor Cad. 71, Bomonti Sisli 34380, Istanbul, Turkey.}
\email{ayse.berkman@msgsu.edu.tr, ayseasliberkman@gmail.com}

\author[A. V. Borovik]{Alexandre V. Borovik }
\address{School of Mathematics,  University of Manchester, Oxford Road, Manchester M13 9PL, UK.}

\date{02 August 2026}

\email{alexandre@borovik.net}

\newtheorem{lemma}{Lemma}[section]
\newtheorem{theorem}[lemma]{Theorem}

\newtheorem{conjecture}[lemma]{Conjecture}
\newtheorem{fact}[lemma]{Assertion}

\newcommand{\acf}{algebraically closed field }

\newcommand{\fmrd}{finite Morley rank}
\newcommand{\fmr}{finite Morley rank }

\usepackage{color}
\definecolor{darkgreen}{rgb}{0,0.6,0}

\begin{document}

\begin{abstract}
This paper contains a {new} version of a final identification theorem for the
`generic' groups of finite Morley rank{, developing further the results of \cite{BB04} and \cite{bbgeneric}}. It is prepared for use in the Hrushovski Programme of classification of {infinite} simple groups of finite Morley rank.
\end{abstract}

\maketitle


\section{Introduction}

\subsection{The Hrushovski Programme}

The Hrushovski Programme is an approach to classification of simple groups of finite Morley rank based on the conjecture by Ehud Hrushovski:

\begin{conjecture}[Hrushovski circa 2003]
If $G$ is an infinite simple group of \fmr and $\sigma$ a generic automorphism of $G$, then the subgroup ${\rm Fix}(\sigma)$ of fixed points of $\sigma$ is pseudofinite.
\end{conjecture}

It is a very strong conjecture; in particular, it kills  {on the spot} the so-called `bad groups', the most notorious class of potential counterexample  \cite{Ugurlu2013}. A detailed discussion of the Hrushovski Programme and its current state can be found in a remarkable recent paper by Piotr Kowalski and P{\i}nar U\u{g}urlu Kowalski \cite{Kowalski-Ugurlu2026}. We quote it:

\begin{quotation}
Hrushovski's suggestion in the paragraph titled `Connection to the Borovik program?' at the end of \cite{Hrushovski2005} can be summarized as follows.
\begin{enumerate}
\item  Take a generic automorphism $\sigma$ of a simple group of finite Morley rank $G$.
\item  Show that the fixed-point subgroup ${\rm Fix}(\sigma)$ is a pseudofinite group. i.e.  a
model of the common theory of all finite groups. [\dots ]
\item  Show that ${\rm Fix}(\sigma) \simeq H(F)$, where $H$ is a simple algebraic group defined
over a pseudofinite field $F$.
\item  Conclude that $G \simeq H(K)$, where $K$ is an algebraically closed field extending $F$.
\end{enumerate}
Hrushovski also points out that even without achieving the crucial Item (2) above:
``counting arguments applied in finite group theory might be applied via $\mu$''.
Here, $\mu$ is a measure on ${\rm Fix}(\sigma)$ constructed in \cite[Proposition 11.1]{Hrushovski2005}  that resembles
the counting measure on a pseudofinite structure. This suggestion has not yet been
pursued.
\end{quotation}

\subsection{The main result of the paper}

A general discussion of {the theory of groups of finite Morley rank}, terminology and notation can be found in the books \cite{bn} and \cite{abc}.

A group of finite Morley rank is said to be of {\em $p'$-type}, if
it contains no infinite abelian subgroup of exponent $p$. Notice
that a simple algebraic group over an \acf $K$ is of $p'$-type if
and only if ${\rm char}\, K \ne p$.


The aim of this work is to {prove} the following version of the generic identification theorems which {appeared} in our earlier work  (\cite{BB04} and \cite{bbgeneric}) {and discuss its potential application in work on the Hrushovski Programme}.

\begin{theorem}[The \c{S}i\c{s} Kebap Theorem]
Let\/ $G$\/ be a simple  group of finite
Morley rank and\/ $D$ a maximal\/ $p$-torus  in $G$
of Pr\"{u}fer rank at least\/ $3$. Let $\alpha$ be an automorphism of the group $G$ normalising $D$. Assume that
\begin{itemize}
\item[{\rm (A)}]  Every proper connected definable subgroup of $G$ which contains $D$  and is normalised by some power $\alpha^k$ of $\alpha$ ($k$ could be different for different subgroups) is a $K$-group.

\item[{\rm (B)}] For every element\/ $x$ of order\/
$p$ in $D$, the group $C^\circ_G(x)$ is of\/ $p'$-type and
$C^\circ_G(x)=F^\circ(C^\circ_G(x))E(C^\circ_G(x))$ (in particular, $C^\circ_G(x)$ is a central product of $F^\circ(C^\circ_G(x))$
and $E(C^\circ_G(x))$).

\item[{\rm (C)}]     $\langle C^\circ_G(x) \mid x \in D, \;|x|=p\rangle = G$.

\end{itemize}
Then $G$ is a Chevalley group over an \acf of characteristic distinct
from
$p$.
\label{sis-kebap}
\end{theorem}

Here, the {\it Fitting subgroup} $F(G)$ is the subgroup of $G$ generated by normal
nilpotent subgroups of $G$, and  $E(G)$ stands for
 the connected component of the product of all
components of $G$; that is the product of all
 quasi-simple subnormal subgroups.

The proof of Theorem~\ref{sis-kebap} is {done} in
Section~\ref{sec:sis-kebap}. Other relevant definitions are given in the same section.

{The previous versions of Theorem~\ref{sis-kebap} are Assertions \ref{sis-kebap2004} and \ref{sis-kebap2011} in Section \ref{sec:history} `Historical remarks'. Potential uses of Theorem~\ref{sis-kebap} are discussed in Sections \ref{how-will-be-used}  and \ref{step-by-step}.}

\subsection{How Theorem~\ref{sis-kebap} could be used.} \label{how-will-be-used}

We believe it could be used at the final stage of projects aimed at  proving theorems of the following kind.
\begin{quote}
\textbf{The Target Theorem.} A  simple group $G$ of finite Morley rank and odd type admitting an automorphism $\alpha$ with a particular property $\mathcal{A}$  is a Chevalley group over an algebraically closed field.
\end{quote}
We intentionally do not specify what  {property $\mathcal{A}$ is} -- it is likely that a few different versions will appear in further work on the Hrushovski Programme. What is highly desirable  {is} that this property $\mathcal{A}$ is \emph{scalable}:
\begin{quote}
\textbf{Scalability.} If $\alpha$ satisfies $\mathcal{A}$ then so do all $\alpha^n$ for $n \in \mathbb{N}$.
\end{quote}

A stronger version of scalability could be even better.

\begin{quote}
\textbf{Strong scalability.} If $H$ is a definable subgroup of $G$ invariant under the action of $\alpha^n$ for some $n \in \mathbb{N}$ then the restriction of $\alpha^n$ on $H$ also has property $\mathcal{A}$.
\end{quote}

Then, if we use induction on the Morley rank of $G$, we
	get the following step-by-step decomposition. It is useful to list steps in the reversed order, from the target to the starting point.

Notice that the property $\mathcal{A}$ in the statement and proof of Theorem~\ref{sis-kebap} became

\medskip
\begin{quote}
\textbf{Generic $\mathcal{A}$}. If $H$ is a proper  {connected} definable subgroup of $G$ which contains  $D$ and  {is} invariant under the action of $\alpha^n$ for some $n \in \mathbb{N}$ then  $H$ is a $K$-group.
\end{quote}

Generic $\mathcal{A}$ is used in {Steps 3--6} of a (potential) proof of the Target Theorem, when we work with groups of Pr\"{u}fer rank at least $3$. However, it is vacuous with smaller Pr\"{u}fer rank cases (that is, in Steps 0--1), where more specific forms of $\mathcal{A}$ are needed.

\subsection{Step by step} \label{step-by-step}

\begin{itemize}
\item[\textbf{Step 6.}]  {Apply} Theorem~\ref{sis-kebap} in its special case  {where} $p=2$  {at} the final stage of  {the} proof of the Target Theorem.
\end{itemize}

\medskip

But this final step should be prepared:

\medskip

\begin{itemize}
\item[\textbf{Step  5.}] Prove the key assumption (B) in Theorem~\ref{sis-kebap}:
\medskip
\begin{quote}
For every involution $x$  in $D$,
\[
C^\circ_G(x)=F^\circ(C^\circ_G(x))E(C^\circ_G(x)).
\]
\end{quote}
\end{itemize}

\medskip

In CFSG, it was called the \emph{$B$-Conjecture}. It should be doable. It looks  like rewriting  Jeff Burdges'  \cite[Theorem 5.1]{Burdges2004} and replacing the assumption that $G$ is a $K^*$-group by a weaker assumption,  Generic  $\mathcal{A}$. This is exactly the same transformation that made our result of 2004 \cite{BB04}  {first into \cite{bbgeneric}, and now into} Theorem~\ref{sis-kebap}.

\begin{fact}
Let $G$ be a simple $K^*$-group of finite  Morley rank and odd type. Then one of the following statements is true:
\begin{itemize}
\item[1.] $n(G)\leq 2$.
\item[2.] $G$ has a proper $2$-generated core.
\item[3.] $G$ satisfies the  $B$-conjecture and contains a classical involution.
\end{itemize}
\end{fact}

For a finite elementary abelian $2$-group $E$, its $2$-{\it rank}
	${m}(E)$ is the minimal number of generators of $E$. For any subgroup $H$, $m(H)$ is the maximum $2$-rank of finite elementary abelian subgroups of $H$.
Let $S$ be a Sylow 2-subgroup of $G$, then the
	{\it normal rank} $n(G)$ is the maximum $2$-rank of
  normal finite elementary abelian subgroups of
		$S$. The {\it 2-generated core} of a group is
	the definable closure of the subgroup generated  by all
	normalizers $N_G(U)$ where $U$ lies in $S$ and has $2$-rank at least $2$. We say that $G$ satisfies the $B$-{\it conjecture} if, for
	any involution $z\in G$,
	$C_G^\circ(z)=F(C_G^\circ(z)) E(C_G(z))$.
 An involution $z$ is
	called {\it classical} if its centralizer contains a
component $A$ such that $z\in Z(A)$ and $A$ is
	isomorphic to $SL_2(K)$ for an algebraically closed field
	$K$.

\begin{itemize}
\item[\textbf{Step  4.}] It becomes clear from the statement of Burdges' Theorem
that we need to get rid of groups with a proper $2$-generated core, and for that purpose rewrite the paper by Borovik, Burdges and Nesin of 2008 \cite{Borovik-Burdges-Nesin2008}. In particular, this is likely to prove assumption (C)
of Theorem ~\ref{sis-kebap}. Indeed, it could be shown that
\[
M= \langle C^\circ_G(x) \mid x \in D, \;|x|=2\rangle \ne G,
\]
then $M$ should be a proper $2$-generated core in $G$.

\end{itemize}

This should be doable.

\medskip
\begin{itemize}
	\item[\textbf{Step  3.}]
 It is likely to be done at Steps  1 and 2, but we have to
ensure that $\alpha$ leaves invariant a maximal $2$-torus in $G$.
\end{itemize}

\medskip
\begin{itemize}
\item[\textbf{Step  2.}] Prove the Target Theorem in the case  {where} the Pr\"{u}fer  {$2$-}rank of $G$ is $2$. Perhaps the hardest bit of any project of this kind.
\end{itemize}

{The Pr\"{u}fer 2-rank of a group of odd type is the  number of copies
	 of the quasicyclic group ${{\mathbb Z}}_{2^\infty}$ in direct product in any of its Sylow 2-subgroups.}

\medskip
\begin{itemize}
\item[\textbf{Step  1.}] Prove the Target Theorem in the case  {where} the Pr\"{u}fer   {$2$-}rank of $G$ is $1$. It is expected that $G$ is  {isomorphic to $\operatorname{PSL}_2(K)$, where $K$ is} an \acf of odd or $0$ characteristic.
\end{itemize}
\medskip

This is the step where property $\mathcal{A}$ should be made, and used, in its specific form.

Within the `supertight project' this was done by
Ulla Karhum\"{a}ki and  P\i nar U\u{g}urlu in \cite[Theorem 1.2]{Karhumaki-Ugurlu2020}:
\begin{fact}[Karhum\"{a}ki and  U\u{g}urlu 2020] \label{th:main} Let $G$ be an infinite simple group of finite Morley rank {of Pr\"ufer $2$-rank $1$}  admitting a supertight automorphism $\alpha$. Assume that the fixed-point subgroup $C_G(\alpha^n)$ is pseudofinite for all $n\in \mathbb{N}\setminus \{0\}$. Then $G \cong {\rm P{S}L}_2(K)$, where $K$ is an algebraically closed field {of odd or $0$ characteristic}.
\end{fact}

\medskip
\begin{itemize}
\item[\textbf{Step  0.}] The Pr\"{u}fer  {$2$-}rank of $G$ is $0$, that is, $G$ is of degenerate type. It should be somehow excluded.
\end{itemize}

\medskip
Again, this is the step where property $\mathcal{A}$ should be made, and used, in its explicit form.

This was done for an automorphism  $\alpha$ which is supertight and all centralisers $C_G(\alpha^n)$, $n \in \mathbb{N}$ are pseudofinite  by P\i nar U\u{g}urlu  in her paper of 2013 \cite[Theorem 3.1]{Ugurlu2013}, see also the last paragraph in \cite[Section 4]{Karhumaki-Ugurlu2020}.

\bigskip

 {The outline of the present paper is as follows. The proof of Theorem~\ref{sis-kebap} (which closely follows \cite{BB04} and is almost the same as \cite{bbgeneric}) is in Section \ref{sec:sis-kebap}.} Some historical observations (starting from  {the first author's} paper of 2001 \cite{berkman}) and a general discussion of work on large projects in mathematics can be found in Section \ref{sec:history}. In particular, the name \c{S}i\c{s} Kebap Theorem will be explained there.

\section{Proof of  Theorem~\ref{sis-kebap}}
\label{sec:sis-kebap}

All definitions {and facts}  related to groups of finite Morley rank in general
can be found in {\cite{abc} and \cite{bn}, to simple algebraic groups -- in \cite{cartk,hump,seitz2}.}

From now on, all groups are assumed to be of finite Morley rank.
The proof follows the proof in \cite{BB04} and \cite{bbgeneric} almost word by word.  In this section, we will point out the places where extra care is needed. In particular, we will introduce a new concept called \emph{trimming}.

The group $G$
is called
a {\it $K$-group}, if every
infinite simple definable and connected section of the group is an
algebraic group over an algebraically closed field.



The strategy is to construct the Weyl group and the root system
of $G$, and then to apply Lyons's Theorem  {\cite[Assertion~2.8]{BB04}}.

From now on, we work under the assumptions of Theorem~\ref{sis-kebap}; that is, $G$ is
a simple group of \fmrd, $D$ is a maximal $p$-torus in $G$
of Pr\"{u}fer rank $\geqslant 3$ and such that every proper definable connected subgroup normalised by some power $\alpha^k$ of $\alpha$ ($k$ could be different for different subgroups) and containing $D$ is a $K$-group. We also assume that
$C^\circ_G(x)$ is  of $p'$-type for every element $x \in D$ of
order $p$, $C^\circ_G(x)=F^\circ(C^\circ_G(x))E(C^\circ_G(x)) $
and
\[G = \langle C^\circ_G(x) \mid x \in D, \; |x|=p\rangle.\]

Denote by $T$ the definable closure of $D$ in $G$. Note $T$ is an $\alpha$-invariant definable divisible abelian subgroup.

\begin{lemma}\label{normalizedbyD} If $M$ is a proper  $\alpha$-invariant definable subgroup in $G$ normalised by $D$, then $M$ is a $K$-group.
	\end{lemma}

\begin{proof}
	Notice that if $M$ is a proper  $\alpha$-invariant definable subgroup in $G$ normalised by $D$ then $MT$ is also a proper  $\alpha$-invariant definable subgroup of $G$ contiaining $D$; for otherwise $M$ would be normal in $G$, which contradicts the simplicity of $G$. Therefore, $MT$ and hence $M$ are  $K$-groups by assumption (A).
	\end{proof}

\subsection{Trimming of the automorphism $\alpha$}

It will be convenient to replace from time to time $\alpha$ by an appropriate power $\alpha^k$. We shall call this manipulation \emph{trimming}. It does not change the assumptions of Theorem \ref{sis-kebap}. {In short, whenever $\alpha$ fixes a finite set setwise, without loss of generality, we may assume that $\alpha$ fixes that set elementwise.}

\textbf{First trimming.} Let $r$ be the Pr\"{u}fer rank of $D$ (and of $T$) then $D$ and $T$ contain $p^r-1$ elements of order $p$. The automorphism $\alpha$ permutes these elements of order $p$ and therefore $\alpha^k$ for $k = (p^r-1)!$ fixes them all. We replace $\alpha$ by $\alpha^k$.

If $x$ is an element of order $p$ in $D$ denote $E_x = E(C^\circ_G(x)).$ Now $E_x$ is $\alpha$-invariant hence is a $K$-group by Lemma~\ref{normalizedbyD}. Each $E_x$ is a central product of finitely many quasisimple algebraic groups over algebraically closed fields (perhaps different), we shall call them components of $E_x$.

{The following lemma is Lemma 3.2 from \cite{BB04} or \cite{bbgeneric}. However, to correct a misprint, we rewrite it here. }

\begin{lemma}
The set $\mathcal{L}$ of components of subgroups $E_x$ for elements of order $p$ in $D$ is non-empty.
\end{lemma}
\begin{proof}
If  $\mathcal{L}$ is empty then all subgroups $E_x =1$ and by condition (C) of Theorem \ref{sis-kebap} and $\langle F^\circ(C^\circ_G(x)) \mid x \in D, \;|x|=p\rangle = G$. But by { Lemma~3.1 in \cite{BB04} or \cite{bbgeneric}}, $D$ is centralised by all subgroups  $F^\circ(C^\circ_G(x))$, hence $D\leqslant Z(G)$, which contradicts the simplicity of $G$.
\end{proof}

\medskip

Obviously, $\mathcal{L}$ is finite and is invariant under the action of $\alpha$, and we can do

\medskip

\textbf{Second trimming.} The automorphism $\alpha$ leaves invariant all subgroups in $\mathcal{L}$.

\subsection{Root Subgroups}

From now on, $SL_2$ will be used instead of $SL_2({\mathbb F})$, etc.

If $L \in \mathcal{L}$ is a component it is normalised by $D$, and $D_L = (D \cap L)^\circ$ is a maximal $p$-torus  in $L$. By properties of simple algebraic groups, the Zariski closure $T_L$ of $D_L$ is a maximal (algebraic) torus in $L$. Root $SL_2$-subgroups associated with the torus $T_L$ are normalised by $D$. Denote by $\Sigma$ the set of all $SL_2$-subgroups associated that way with the $p$-torus $D$. There are finitely many of them, and $\alpha$ leaves the set $\Sigma$ invariant, so we do

\textbf{Third trimming.} The automorphism $\alpha$ leaves invariant all $SL_2$-subgroups in $\Sigma$.

They are our future root
$SL_2$-subgroups.

 The key subgroups appearing here are generated by subgroups from $\Sigma$ and are therefore definable, connected,  $\alpha$-invariant, and are normalised by $D$; and hence they are $K$-groups by Lemma~\ref{normalizedbyD}.

Now the proof follows {Lemmas 3.3 to 3.6} in \cite{BB04} and \cite{bbgeneric} almost verbatim.
{The only exception is that} one has to be careful about whether certain important subgroups are still $K$-groups in this new context as well. For example, any proper subgroup generated by elements of $\Sigma$ is a $K$-group.

\subsection{Weyl Group}

{For} $L\in\Sigma$, {let} $H_L$ stand for the maximal
algebraic torus $H_L:=C_L(D\cap L)$ in $L\cong SL_2$. Now set $H =
\langle H_L \mid L \in \Sigma \rangle$ and call it the {\em
natural torus associated with $D$}.

For any $L \in \Sigma$, $W(L):=N_L(H)H/H=N_L(H_L)/H_L$ is the
Weyl group of $SL_2$ and has order $2$; hence $W(L)$ contains a
single involution, which will be denoted by $r_L$.

Notice that $D$ is a $p$-torus, the subgroups $N_G(D)$ and
$C_G(D)$ are definable and the factor group $N_G(D)/C_G(D)$ is
finite. Set $W:=N_G(D)/C_G(D)$.

\textbf{Fourth trimming.} {Note that $\alpha$ fixes $W$, and since $W$ is finite, we may assume  by trimming that $\alpha$ fixes every element in $W$.}

This will guarantee that subgroups like $\langle H,L,r_L\rangle$ for  $L\in \Sigma$ are also $K$-groups.

{
The end of the proof is slightly unusual:
\begin{quote}
Reading Lemmas {3.7} through 3.17 in \cite{BB04} or \cite{bbgeneric}, with these remarks in mind will prove Theorem~\ref{sis-kebap}.
\end{quote}
Our aim was to make  ideas and arguments from \cite{BB04} and \cite{bbgeneric} to  work in a completely different context and this achieves the aim. \hfill $\Box$
}

\section{Historical remarks} \label{sec:history}

\subsection{The Hrushovski Programme}

The second author first heard about {the suggestions of Hrushovki that} later became known {as} the Hrushovski Programme from Gregory Cherlin in a conversation which took place on 10 July 2003 in Lyon. Here is an entry in his diary:

\begin{quote}
Greg told me that Hrushovski said that we {have the} wrong approach to groups of \fmrd. One has to take a generic automorphism of a saturated model of $G$ and study the structure formed by generic points of ${\rm Fix}(\sigma)$, it has properties close to those of  a pseudofinite structure, in particular, there are more precise measures of definable sets.
\end{quote}

\subsection{The Generic Identification Theorem}

In  {the} classification  {project of infinite simple} groups of \fmrd, it was realised relatively early   {that the} identification of `generic' simple groups was not possible by proving directly that they were algebraic groups, that is, algebraic varieties over algebraically closed fields endowed with a group structure that is compatible with its structure as an algebraic variety.

Instead, they had to be characterised as groups with a $BN$-pair (this approach was used in \cite{abc} in the case of groups of even type) or identified with Chevalley groups by means of  a properly developed analogue of the Curtis--Phan theorem from  finite group theory. Both approaches were {shown to work (under certain conditions) by the first author} in her {doctoral} thesis of 1998 \cite{berkmanthesis}.

Lyons's Theorem (Assertion~2.8 \cite{BB04}) which generalised the Curtis--Phan theorem from finite groups to Chevalley groups  over arbitrary fields provided a powerful tool  {for} `generic identification'. Its first appearance was in  {the first author's} Classical Involution Theorem of 2001 \cite{berkman}.

A compact formulation of a Generic Identification Theorem suitable for application in groups of  \fmr  of both even and odd type was given by us in 2004 in \cite{BB04}.

\begin{fact} {\rm \cite[Theorem 1.2]{BB04}} Let $G$ be a simple  $K^*$-group of finite
	Morley rank, $p$ a prime, and $D$ a maximal $p$-torus  in $G$ of Pr\"{u}fer rank at least $3$. Assume that
\begin{enumerate}
  \item For every element $x$ of order
	$p$ in $D$, the group $C^\circ_G(x)$ is of $p'$-type and
  $C^\circ_G(x)=F^\circ(C^\circ_G(x))E(C^\circ_G(x))$.
  \item $G=\langle C^\circ_G(x) \mid x \in D, \;|x|=p\rangle$.
\end{enumerate}
Then $G$ is a Chevalley group over an \acf of characteristic not $p$.
\label{sis-kebap2004}
\end{fact}
We called it {the} \emph{\c{S}i\c{s} Kebap  Theorem} because the  maximal $p$-torus $D$ played {the} role of a  skewer (\c{s}i\c{s}) with {root $\operatorname{SL_2}$-sub}groups strung on it like cubes of {meat or tofu}.

When in 2011, we started our project on primitive permutation groups of \fmr and affine type (a detailed survey of that project spread over papers  \cite{bbgeneric,BB11,bbsharp,bbnotsharp,bbsolvable} and the guiding principles of its development from the target to the current situation, from the top to bottom, can be found in a recent paper \cite{BB25}), it became clear for us, that we could not expect that the group under our investigation was a $K^*$-group. On the other hand, we knew that we could get a good control of its maximal torus. For that reason, we started our work by appropriately modifying the \c{S}i\c{s} Kebap  Theorem.

\begin{fact}[Berkman and Borovik  2011, \cite{bbgeneric}]
Let\/ $G$\/ be a simple  group of finite
Morley rank and\/ $D$ a maximal\/ $p$-torus  in $G$
of Pr\"{u}fer rank at least\/ $3$. Assume that
\begin{itemize}
\item[{\rm (A)}]  Every proper connected definable subgroup of $G$ which contains $D$ is a $K$-group.

\item[{\rm (B)}] For every element\/ $x$ of order\/
$p$ in $D$, the group $C^\circ_G(x)$ is of\/ $p'$-type and
$C^\circ_G(x)=F^\circ(C^\circ_G(x))E(C^\circ_G(x)) $.

\item[{\rm (C)}]     $\langle C^\circ_G(x) \mid x \in D, \;|x|=p\rangle = G$.

\end{itemize}
Then $G$ is a Chevalley group over an \acf of characteristic distinct
from
$p$.
\label{sis-kebap2011}
\end{fact}

As one can easily see, Theorem \ref{sis-kebap} is a direct descendant  of Assertion \ref{sis-kebap2011}, with a large chunk of its proof being reproduced \emph{verbatim}.

\subsection{An example from finite group theory}

We cannot recall any result in finite group theory where the study of a group  with a special automorphism required reusing and rewriting a large chunk of CFSG, with a single exception: a remarkable result by Paul Flavell \cite{Flavell2006} where he put the final full stop in the celebrated Hall--Higman theory.

His result has a surprisingly simple formulation:

\begin{quote}
Theorem A. Suppose that:
\begin{itemize}
\item $R$ is a group of prime order $r$ that acts on the $r'$-group $G$.
\item $V$ is a faithful irreducible $RG$-module over a field of characteristic $p$.
\item $C_V (R)$ = 0.
\end{itemize}
Then either $[G, R] = 1$ or $r$ is a Fermat prime and $[G, R]$ is a nonabelian
special $2$-group.
\end{quote}

See, for example, Section 5 `Strong embedding', and the first remark on page 373 in \cite{Flavell2006}:

\begin{quote}
(a) [the 2-transitivity of a certain action of $G$] is a special case of a result of Aschbacher and Bender. We follow  closely the proof given in \cite{gls4}. Our more restrictive hypothesis allows us to truncate the argument at an early stage.
\end{quote}

We may find ourselves in a similar situation while developing different streams in the Hrushovski Programme: being in want of reusing huge material accumulated in the (unfinished) classification of infinite simple groups of finite Morley rank while hoping that ``our more restrictive hypothesis allows us to truncate the argument at an early stage.''

Of course a proof of the pseudofinitness of the group of fixed points of a generic automorphism of a simple group of \fmr will provide a dramatic shortcut and will make this approach largely unnecessary.

\section*{Acknowledgements}

We wish to express our thanks to our colleagues Gregory Cherlin, Adrien Deloro, Ulla Karhum\"{a}ki, Piotr Kowalski,  P\i nar U\u{g}urlu Kowalski and \c{S}\"ukr\"u Yal\c{c}\i nkaya for many frutiful discussions.

\end{document}